\documentclass[10pt]{article}

\usepackage{amsmath, amssymb, amsxtra, mathrsfs}
\usepackage{graphics, epsfig, here, epstopdf, fancyhdr,graphicx}
\usepackage{bbm}
\usepackage{color}
\usepackage{multirow}
\usepackage[all]{xy}

\newcommand{\bR}{\mathbb{R}}

\newcommand{\cE}{\mathcal{E}}

\newcommand{\cJ}{\mathcal{J}}
\newcommand{\cL}{\mathcal{L}}

\newcommand{\cV}{\mathcal{V}}

\newcommand{\fA}{\,\forall\,}

\newcommand{\dS}{\displaystyle}

\newcommand{\pp}{\partial}

\newcommand{\ppx}[1]{{\frac{\partial }{\partial #1}}}

\usepackage{latexsym}
\usepackage{graphics}
\usepackage{xspace}
\usepackage{psfrag}
\usepackage{epsfig}
\usepackage{pst-all}
\usepackage{amssymb, amsmath, amsthm, verbatim, mathrsfs}
\usepackage{bbm}

\usepackage{algorithmic}
\usepackage{epstopdf}

\usepackage{amsmath} 

\title{\LARGE \bf
Monotone Order Properties for Control of \\ Nonlinear Parabolic PDE on Graphs
}

\author{Sidhant Misra, Marc Vuffray, Anatoly Zlotnik and Michael Chertkov
\thanks{S. Misra, M. Vuffray, A. Zlotnik and M. Chertkov are with Los Alamos National Laboratory, Los Alamos, NM 87545.  \qquad
Email: \{ sidhant $\mid$ vuffray $\mid$ azlotnik $\mid$ chertkov \}@lanl.gov.}}

\newtheorem{theorem}{Theorem}
\newtheorem{lemma}{Lemma}

\newtheorem{corollary}{Corollary}
\newtheorem{definition}{Definition}

\newtheorem{assumption}{Assumption}

\begin{document}

\maketitle

\begin{abstract}
We derive conditions for the propagation of monotone ordering properties for a class of nonlinear parabolic partial differential equation (PDE) systems on metric graphs.  For such systems, PDEs with a general nonlinear dissipation term define evolution on each edge, and balance laws create Kirchhoff-Neumann boundary conditions at the vertices.  Initial conditions, as well as time-varying parameters in the coupling conditions at vertices, provide an initial value problem (IVP).  We first prove that ordering properties of the solution to the IVP are preserved when the initial conditions and time-varying coupling law parameters at vertices are appropriately ordered.  Then, we prove that when monotone ordering is not preserved, the first crossing of solutions occurs at a graph vertex.  We consider the implications for robust optimal control formulations and real-time monitoring of uncertain dynamic flows on networks, and discuss application to subsonic compressible fluid flow with energy dissipation on physical networks.
\end{abstract}


\section{Introduction} \label{secintro}

The preservation of monotone order propagation (MOP) properties in dynamical systems has been extensively investigated in the context of ordinary differential equation theory \cite{kamke32,hirsch85,smith88,hirsch05}.  The recent discovery of numerous applications has renewed interest in such systems, for example to vehicle routing under uncertainty \cite{como10}, analysis of chemical reaction networks \cite{deleenheer04}, as well as power systems and turbulent jet flows \cite{budivsic12}.  The notion of monotone control systems \cite{angeli03,angeli04} has also facilitated stability analysis for systems with MOP properties \cite{lovisari14}, and enabled robust control in applications including automation of building ventilation systems \cite{meyer13}.  Several results on the propagation of order properties for stochastic systems exist as well \cite{sootla15a}.

Previous studies on monotone dynamical systems have largely focused on MOP properties of ordinary differential equations (ODEs) \cite{hirsch05}, and applications involving representations of fluid flow or the aggregated motion of discrete particles were examined with ODE models \cite{hirsch05,como13a}.  However, control and optimization approaches to systems represented by PDE dynamics could benefit significantly from monotone systems concepts, in particular control of fluid flows on networks \cite{steinbach07pde,zlotnik15cdc} and quantum graphs \cite{arioli16}.  While recent results have used monotonicity properties to optimize fluid flows over networks using set-theoretic and variational approaches \cite{vuffray15cdc,dvijotham15}, these focused on the steady-states of the flow equations.  These studies demonstrate that the steady-states have a monotone ordering with respect to certain input parameters.  Crucially, this property was shown to enable significant simplification of robust optimization formulations, in particular for distributed flows on large-scale networks.

The need to develop robust optimal control formulations for emerging applications involving uncertain dynamic flows on networks motivates investigation of MOP properties for PDEs.
The approximation of a diffusive PDE operator by an ODE system and derivation of MOP properties using the established ODE theory has been suggested for basic reaction-diffusion problems \cite{deleenheer04,enciso06}.  Otherwise, monotone operators have been examined primarily in the context of existence and approximations of solutions to nonlinear PDE systems \cite{quaas08,briani12,showalter13}.  Recently, conditions for MOP properties were derived for actuated dynamic commodity flows through networks 
\cite{zlotnik16ecc}.  The notion of a monotone parameterized control system was introduced, and MOP properties were shown to facilitate efficient formulation of robust optimal control problems with uncertainty in nodal commodity withdrawals.  Lumped-element approximation was used to express the dissipative PDEs on network edges as ODE systems, to which existing MOP theory was applied.  However, no {\em ab initio} analysis of MOP properties of PDE systems on graphs has been performed to date.

In this manuscript, we derive several results on the propagation of monotone order properties for systems of nonlinear parabolic PDEs on metric graphs.  Specifically, PDEs with a general nonlinear dissipation term define state evolution on each edge, and balance laws create Kirchhoff-Neumann boundary conditions at the vertices.  We first suppose that initial conditions, together with time-varying parameters that characterize coupling conditions at vertices, provide a well-posed initial value problem (IVP).  Our main result is a theorem establishing preservation of monotone ordering properties of the solution to the IVP when the initial conditions and time-varying coupling law parameters at vertices are appropriately ordered.  Furthermore, we prove that when monotone ordering is not preserved, the first crossing of solutions occurs at a graph vertex.

The manuscript is organized as follows.  In Section \ref{sec:formulation}, we formulate a class of nonlinear parabolic PDE systems defined on a collection of domains that form a metric graph when coupled by nodal Kirchhoff-Neumann boundary conditions, and state the main results given the required assumptions.  Section \ref{sec:proof} contains the formal proofs of the main results on monotone order propagation and crossing point condtions for solutions to the PDE system.  Then, implications for formulating robust optimal control problems and monitoring policies for uncertain dynamic flows on networks are discussed in Section \ref{sec:application}, followed by a review of applications to subsonic compressible fluid flow with energy dissipation on physical networks.  We summarize our conclusions in Section \ref{sec:conc}.


\vspace{-1ex}
\section{Parabolic PDE Systems on Metric Graphs} \label{sec:formulation}

We consider a metric graph $\Gamma=\left(\cV,\cE,\lambda\right)$ where $\cV$ is the set of vertices and $\cE\subset \cV \times \cV$ is the set of directed  edges $(i,j)\in\cE$ that connect the vertices $i,j\in\cV$.  Here $\lambda:\cE\to\bR_+$ is a metric on the edges, where $\bR_+$ denotes the non-negative real numbers. Let the incoming and outgoing neighborhoods of $j\in \cV$ be denoted by $\partial_{+}j$ and $\partial_{-}j$, respectively.  These sets are defined as
\begin{align}
\partial_{+}j&=\left\{ i\in \cV\mid(i,j)\in \cE\right\} \\
\partial_{-}j&=\left\{ k\in \cV\mid(j,k)\in \cE\right\}.
\end{align}
Every edge $(i,j)\in \cE$ is associated with a spatial dimension on the interval $I_{ij}=[0,L_{ij}]$, where $L_{ij}=\lambda(i,j)>0$ is interpreted as the edge length defined by the metric $\lambda$.  We let $V=|\cV|$ and $E=|\cE|$ denote the number of vertices and of edges, respectively.

The state of the network system is characterized within each edge $(i,j)\in\cE$  by space-time dependent variables corresponding to flow $\phi_{ij}:[0,T]\times I_{ij}\rightarrow\mathbb{R}$ and non-negative density $\rho_{ij}:[0,T]\times I_{ij}\rightarrow\mathbb{R}_{+}$. In addition, every vertex $i\in \cV$ is associated with a time-dependent {internal} nodal density $\rho_{i}(t):[0,T]\rightarrow\mathbb{R}_{+}$ and is subject to
a time-dependent flow injection $q_{i}:[0,T]\rightarrow\mathbb{R}$.

We suppose that the density and flow dynamics on the edge  $(i,j)\in\cE$ evolve according to the generalized dissipative relations
\begin{align}
\dS \pp_t\rho_{ij}(t,x_{ij})+\pp_x\phi_{ij}(t,x_{ij}) & =  0 \label{eq:in_continuity} \\
\phi_{ij}(t,x_{ij})+f_{ij}(t,\rho_{ij}(t,x_{ij}), \partial_{x}\rho_{ij}(t,x_{ij})) & =0,  \label{eq:in_dissipation_eq}
\end{align}
which are called respectively the continuity and momentum dissipation equations.  


Next, we establish nodal relations that characterize the boundary conditions for the flow dynamics \eqref{eq:in_continuity}-\eqref{eq:in_dissipation_eq} on each edge of the graph.  For this purpose, in order to simplify notation we define
\begin{align}
\underline{\rho}_{ij}(t)\triangleq\rho_{ij}(t,0), \quad \overline{\rho}_{ij}(t)\triangleq\rho_{ij}(t,L_{ij}), \label{eq:end_p_def} \\
\underline{\phi}_{ij}(t)\triangleq\phi_{ij}(t,0), \quad \overline{\phi}_{ij}(t)\triangleq\phi_{ij}(t,L_{ij}). \label{eq:end_q_def}
\end{align}
At each vertex $i\in V$ the flow and density values at the endpoints of adjoining edges must satisfy certain compatibility conditions.  First, a Kirchhoff-Neumann property of flow conservation is ensured through nodal continuity equations
\begin{align}
q_j(t)+\sum_{i\in\partial_{+}j}\overline{\phi}_{ij}- \sum_{k\in\partial_{-}j}\underline{\phi}_{jk}=0, \quad \fA j\in\cV. \label{eq:in_nodal_continuity}
\end{align}
In addition, we include compatibility conditions that relate nodal densities to boundary conditions on edges.  
For each edge $(i,j)\in\cE$, the corresponding nodal conditions are
\begin{align}
\underline{\rho}_{ij}(t) = \underline{\alpha}_{ij} (t,\rho_{i}(t)), \quad
\overline{\rho}_{ij}(t)  = \overline{\alpha}_{ij}(t,\rho_{j}(t)), \label{eq:in_pressure_comp}
\end{align}
where the compatibility functions $\underline{\alpha}_{ij}(t,\rho)$ and $ \overline{\alpha}_{ij}(t,\rho)$ are monotonically increasing functions in $\rho$ for all $t\in[0,T]$ and $\rho>0$.
The functions $\rho_{i}$ are auxiliary variables that denote internal nodal density values.  The above compatibility conditions are visualized in Figure \ref{fig:compatibility}.

We suppose that instantaneous state of the system at time $t=0$ is specified by initial density and flow profiles
\begin{align}
\!\!\! \rho_{ij}(0,x)=\rho_{ij}^{0}(x), \,\, \phi_{ij}(0,x)=\phi_{ij}^{0}(x), \quad \fA (i,j)\in\cE. \label{eq:in_initial_condition}
\end{align}

\begin{figure}
\centering
\includegraphics[width=.95\linewidth]{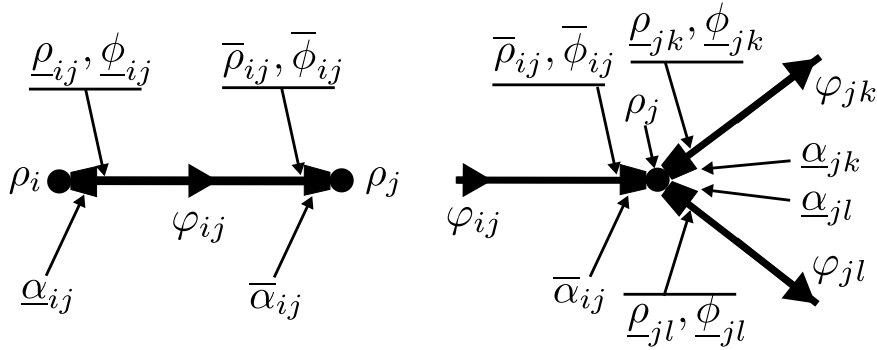}
 \caption{Nodal densities $\rho_j$ and boundary variables $\underline{\rho}_{ij}$, $\underline{\phi}_{ij}$, $\overline{\rho}_{ij}$, and $\overline{\phi}_{ij}$, and compatibility functions $\underline{\alpha}_{ij}$ and $\overline{\alpha}_{ij}$
for an edge (left) and a joint (right).
} \label{fig:compatibility}
\end{figure}

%

\begin{assumption}
 \label{assume}
We make the following assumptions on initial value problem \eqref{eq:in_continuity}-\eqref{eq:in_initial_condition} that describes the coupled network flow dynamics with initial conditions.
\begin{itemize}
\item[(i)] \emph{Well-posedness and regularity of initial conditions:} There exists an integer $k\geq 2$ such that $\rho_{ij}^{0},\,\phi_{ij}^{0}\in C^k([0,L_{ij}])$ for all $(i,j)\in\cE$. Moreover the coupling constriants \eqref{eq:in_nodal_continuity} and \eqref{eq:in_pressure_comp} hold at $t=0$.
\item[(ii)] \emph{Continuity of inputs and control:} The compatibility functions satisfy $\underline{\alpha}_{ij},\,\overline{\alpha}_{ij}\in C_+^k([0,T]\times\bR_+)$ for all $(i,j)\in\cE$,
and the nodal parameter functions satisfy $q_i\in C^k([0,T])$ for all $i\in\cV$.
\item[(iii)] \emph{Well-posedness of coupled network dynamics:} The initial value problem consisting of the coupled network flow dynamics with the initial conditions in \eqref{eq:in_continuity}-\eqref{eq:in_initial_condition}, along with given
compatibility functions $\overline{\alpha}_{ij}$ and $\underline{\alpha}_{ij}$, admits a unique classical solution that is twice continuously differentiable.
\item[(iv)] \emph{Stability under small perturbations:} Let $\rho_{ij}(t,x_{ij}), \phi_{ij}(t,x_{ij})$ for $(i,j) \in \cE$ be the unique classical solution to \eqref{eq:in_continuity}-\eqref{eq:in_initial_condition}. Let
$\rho_{ij,\epsilon}(t,x_{ij})$ and $\phi_{ij,\epsilon}(t,x_{ij})$ for all $(i,j) \in \cE$ be a solution to the perturbed system
\begin{align}
\!\!\!\!\!\!\!\!\!   \dS \pp_t\rho_{ij,\epsilon}(t,x_{ij})\!+\!\pp_x\phi_{ij,\epsilon}(t,x_{ij})  - \epsilon & =  0,  \label{perturbed1} \\
\!\!\!\!\!\!\!\!\!  \phi_{ij,\epsilon}(t,x_{ij})\!+\!f_{ij}(t,\rho_{ij,\epsilon}(t,x_{ij}), \partial_{x}\rho_{ij,\epsilon}(t,x_{ij}))& =0,  \label{perturbed2}
\end{align}
with the perturbed initial conditions
\begin{align}
\!\!\! \rho_{ij,\epsilon}(0,x)=\rho_{ij}^{0}(x) + \epsilon, \quad \phi_{ij,\epsilon}(0,x)=\phi_{ij}^{0}(x)  \label{eq:in_initial_condition_pert}
\end{align}
for all $(i,j)\in\cE$.
Then as $\epsilon \rightarrow 0$, the perturbed solution converges point-wise to the original solution, i.e., for all $(i,j) \in \cE$, $\ x_{ij} \in I_{ij}$  and  $t \in [0,T]$, we have
\begin{align}
\lim_{\epsilon \rightarrow 0} \rho_{ij,\epsilon}(t,x_{ij}) = \rho_{ij}(t,x_{ij}).
\end{align}
\end{itemize}
\end{assumption}



\vspace{0.1in}

\begin{theorem} \label{theorem}
Suppose the initial value problem described in \eqref{eq:in_continuity}-\eqref{eq:in_initial_condition} satisfies Assumption \ref{assume}. Also suppose that the dissipation function
$f_{ij}(t,u,v)$ is strictly increasing in the third argument $v$ for all $(i,j) \in \cE$. Let $\rho_{ij}^{(1)}(0,x_{ij})$ and  $\rho_{ij}^{(2)}(0,x_{ij})$ be two initial conditions that satisfy $\rho_{ij}^{(1)}(0,x_{ij}) \geq \rho_{ij}^{(2)}(0,x_{ij})$ for  all $(i,j) \in \cE$, $x_{ij} \in I_{ij}$. Let $\mathcal{S} \subseteq \cV$ be an arbitrary subset of $\cV$. Let $t_0 \in [0,T]$ and suppose that for all $i \in \mathcal{S}$ we have that $q_i^{(1)}(t) \geq q_i^{(2)}(t)$
for all $t \in [0,t_0]$ and for all $i \in \cV \setminus \mathcal{S}$ we have that $\rho_i^{(1)}(t) \geq \rho_i^{(2)}(t)$ for all $t \in [0,t_0]$. Then  the densities in the system satisfy $\rho_{ij}^{(1)}(t,x_{ij}) \geq \rho_{ij}^{(2)}(t,x_{ij})$
for all $(i,j) \in \cE$, $x_{ij} \in I_{ij}$ and $t \in [0,t_0]$.

\end{theorem}

\vspace{0.1in}
Observe that that no assumption is made in Theorem~{\ref{theorem}} regarding the nodal parameter functions $q_i^{(1)}(t)$ and  $q_i^{(2)}(t)$ for nodes $i\in V\setminus \mathcal{S}$. The statement implies that if we start with two initial conditions that satisfy a certain
ordering, then regardless of the nodal parameter functions, if this ordering is ever violated during the course of the evolution of the system, the violation must first occur at one of the vertices $i\in V \setminus \mathcal{S}$ of the network. As a consequence, if the nodal density values $\rho_{i}(t)$ for nodes $i \in V \setminus \mathcal{S}$ also satisfy the same ordering, then the ordering of the initial conditions is preserved throughout the evolution of the system.

\section{Proof of Main Result} \label{sec:proof}
\subsection{Crossing Points}
The proof is constructed by establishing the non-existence of the so-called ``first crossing point". We formalize this definition below.
\begin{definition} \label{def:first_crossing}
 Let $\rho_{ij}^{(1)}(t,x_{ij})$ and $\rho_{ij}^{(2)}(t,x_{ij})$ be the unique classical solutions corresponding to the initial conditions $\rho_{ij}^{(1)}(0,x_{ij})$ and  $\rho_{ij}^{(2)}(0,x_{ij})$ and injections $q^{(1)}_i(t)$ and $q^{(2)}_i(t)$
 respectively. Further suppose that for all $(i,j) \in \cE$ and for all $x_{ij} \in I_{ij}$ we have $\rho_{ij}^{(1)}(0,x_{ij}) \geq \rho_{ij}^{(2)}(0,x_{ij})$. Then a tuple $(t_c, x_c)$, where $t_c \in (0,T]$ and $x_c \in I_{ij}$ for some $(i,j) \in \cE$ is called a first crossing point if
 \begin{align}
 t_c = \sup \{t \in &[0,T] \ : \ \rho_{ij}^{(1)}(t,x_{ij}) \geq \rho_{ij}^{(2)}(t,x_{ij}) \nonumber  \\
 & \fA (i,j) \in \cE, \, x_{ij} \in I_{ij}  \}, \label{eq:crossing_time}
 \end{align}
 and, there exists a $\delta > 0$, such that
 \begin{align}
 \rho_{ij}^{(1)}(t,x_{c}) < \rho_{ij}^{(2)}(t,x_{c}),  \label{eq:positive_derivative}
 \end{align}
 for all $t \in (t_c, t_c + \delta)$.
\end{definition}

First crossing points need not be unique because there may be multiple coordinates $x_c$ that satisfy the above definition. The crossing time $t_c$ however, is unique by definition. Note that whenever there is no crossing point in the system dynamics until some time $t_0$, then the ordering of the initial conditions must be preserved till $t_0$. In the rest of the section, we prove the appropriate non-existence of first crossing points in order to establish Theorem~\ref{theorem}.

\subsection{Technical Lemmas}
In this section, we prove several technical lemmas that will be useful to establish Theorem~\ref{theorem}. Let $\rho_{ij, \epsilon}^{(1)}(t,x_{ij})$ be the solution to the perturbed system in  \eqref{perturbed1}-\eqref{perturbed2}
with nodal input parameters set at $q_{i}^{(1)}(t)$. The following lemmas prove that there can be no crossing point of the perturbed solution $\rho_{ij, \epsilon}^{(1)}(t,x_{ij})$ and $\rho_{ij}^{(2)}(0,x_{ij})$ either in the interior of an edge or at a vertex $i \in \cV$ where $q_i^{(1)}(t) \geq q_i^{(2)}(t)$.

\begin{lemma} \label{lem:perturbed_interior}
Let $(i,j) \in \cE$ be an edge. Suppose that for all $x_{ij} \in I_{ij}$ we have $\rho_{ij}^{(1)}(0,x_{ij}) \geq \rho_{ij}^{(2)}(0,x_{ij})$.
Then there is no first crossing point  $(t_c, x_c)$ between the perturbed solution $\rho_{ij, \epsilon}^{(1)}(t,x_{ij})$ and its non-perturbed counterpart $\rho_{ij}^{(2)}(t,x_{ij})$ such that $t_c \in [0,T]$ and $0 < x_c < L_{ij}$.
\end{lemma}

\begin{lemma} \label{lem:nodal_cross_equiv}
Let $j \in \cV$.  Suppose that a first crossing occurs at $(t_c, x_c=0)$ on an edge $(j,k)\in\cE$ for one $k\in\partial_- j$ or at $(t_c, x_c=L_{ij})$ on an edge $(i,j)\in\cE$ for one $i\in\partial_+ j$.  Then a first crossing point occurs at $(t_c, x_c=0)$ for all edges $(j,k)\in\cE$ with $k\in\partial_- j$ and also at $(t_c, x_c=L_{ij})$ for all edges $(i,j)\in\cE$ with $i\in\partial_+ j$.
\end{lemma}

\begin{lemma} \label{lem:perturbed_vertex}
Let $j \in \cV$. Suppose that for all $k \in \partial_- j$ and for all $x_{jk} \in I_{jk}$ we have $\rho_{jk}^{(1)}(0,x_{jk}) \geq \rho_{jk}^{(2)}(0,x_{jk})$. Further suppose that $q_j^{(1)}(t) \geq q_j^{(2)}(t)$.
Then there is no first crossing point $(t_c, x_c)$ between the perturbed solution $\rho_{jk, \epsilon}^{(1)}(t,x_{jk})$ and $\rho_{jk}^{(2)}(t,x_{jk})$ such that $t_c \in [0,T]$ and $x_c = 0$ for any $k\in\partial_-j$.
\end{lemma}

\begin{proof}[Proof of Lemma~\ref{lem:perturbed_interior}]
Suppose for the sake of contradiction that there exists a first crossing point $(t_c, x_c)$ such that $0  < x_c < L_{ij}$. Then by Definition \ref{def:first_crossing},
\begin{align}
\rho_{ij, \epsilon}^{(1)}(t_c, x_c) &= \rho_{ij}^{(2)}(t_c, x_c), \label{cp_1} \\
\rho_{ij, \epsilon}^{(1)}(t_c, x) &\geq \rho_{ij}^{(2)}(t_c, x), \ x \in (0,L_{ij}). \label{cp_2}
\end{align}
By Assumption \ref{assume}, the functions $\rho_{ij, \epsilon}^{(1)}(t_c, x)$ and $\rho_{ij}^{(2)}(t_c, x)$ and hence the function $g: (0,L_{ij}) \rightarrow \mathbb{R}$ given by
\begin{align} \label{eq:g_fun}
g(x) &= \rho_{ij, \epsilon}^{(1)}(t_c, x)  - \rho_{ij}^{(2)}(t_c, x)
\end{align}
is in $C^k$. Combined with \eqref{cp_1}-\eqref{cp_2}, this means the function $g(.)$ must also satisfy
\begin{align}
\ppx{x}g(x_c) = 0, \quad \frac{\partial}{\partial x^2}g(x_c) \geq 0, \label{cp_3}
\end{align}
which in turn yields
\begin{align}
\partial_x \rho_{ij, \epsilon}^{(1)}(t_c, x_c) &= \partial_x\rho_{ij}^{(2)}(t_c, x_c), \label{eq:cp_equality} \\
 \partial_x^2 \rho_{ij, \epsilon}^{(1)}(t_c, x_c) &\geq \partial_x^2 \rho_{ij}^{(2)}(t_c, x_c). \label{eq:cp_ineq1}
\end{align}
See Figure~\ref{fig:crossing} for a pictorial interpretation of the relations \eqref{eq:cp_equality} and \eqref{eq:cp_ineq1}.

\begin{figure}[t]
\centering
\includegraphics[width=1.0\linewidth]{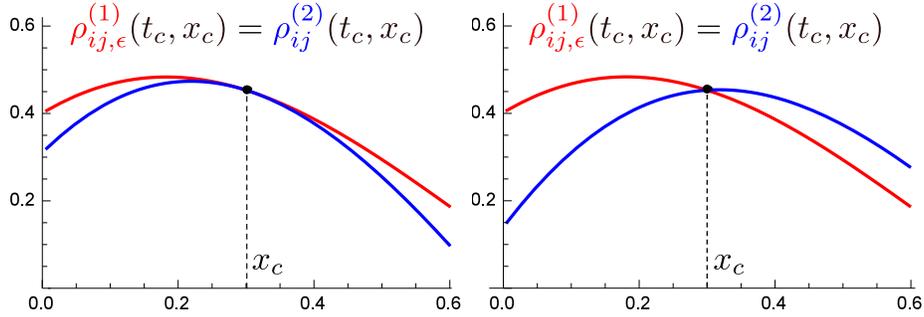} \vspace{-2ex} \caption{Left Figure: example of a first crossing point at $x_c$. Right Figure: example of a crossing point at $x_c$ that is not a first crossing point.  Relations \eqref{eq:cp_equality} and \eqref{eq:cp_ineq1} are satisfied in the left Figure but not in the right Figure.
} \label{fig:crossing} \vspace{2ex}
\end{figure}

Further, \eqref{eq:positive_derivative} implies that $\partial_t \rho_{ij, \epsilon}^{(1)}(t_c, x_c) \leq \partial_t \rho_{ij}^{(2)}(t_c, x_c)$, so that applying the continuity equation \eqref{eq:in_continuity} and its perturbed version \eqref{perturbed1}, we obtain
\begin{align}
-  \pp_x\phi_{ij,\epsilon}^{(1)}(t_c,x_c)  + \epsilon \leq - \pp_x\phi_{ij}^{(2)}(t_c,x_c). \label{eq:pert_comp}
\end{align}
We substitute for the flow terms in \eqref{eq:pert_comp} using the dissipation equation \eqref{eq:in_dissipation_eq} and its perturbed counterpart \eqref{perturbed2} to obtain the relation
\begin{align}
 \partial_x &f_{ij}(t_c,\rho_{ij, \epsilon}^{(1)}(t,x_{c}), \partial_{x}\rho_{ij, \epsilon}^{(1)}(t,x_{c})) \nonumber \\
   &\leq \partial_x  f_{ij}(t,\rho_{ij}^{(2)}(t,x_{c}), \partial_{x}\rho_{ij}^{(2)}(t,x_{c})) - \epsilon \nonumber \\
   &<  \partial_x  f_{ij}(t,\rho_{ij}^{(2)}(t,x_{c}), \partial_{x}\rho_{ij}^{(2)}(t,x_{c})). \label{eq:cp_ineq2}
\end{align}
Using the chain rule for differentiation, we can rewrite for $k = 1,2$,
\begin{align}
\!\!\! \partial_x &f_{ij}(t_c,\rho_{ij}^{(k)}(t,x_{c}), \partial_{x}\rho_{ij}^{(k)}(t,x_{c}))  \nonumber \\
\!\!\! = & \,\partial_u f_{ij}(t_c,\rho_{ij}^{(k)}(t,x_{c}), \partial_{x}\rho_{ij}^{(k)}(t,x_{c})) \partial_x \rho_{ij}^{(k)}(t,x_{c}) \nonumber\\
\!\!\!  &+  \partial_v f_{ij}(t_c,\rho_{ij}^{(k)}(t,x_{c}), \partial_{x}\rho_{ij}^{(k)}(t,x_{c})) \partial_x^2 \rho_{ij}^{(k)}(t,x_{c}). \label{eq:fx_ineq0}
\end{align}
Substituting \eqref{eq:fx_ineq0} into \eqref{eq:cp_ineq2} and using \eqref{cp_1} and \eqref{eq:cp_equality}, we get
\begin{align}
\!\!\!\!\!\!\!& \partial_v f_{ij}(t_c,\rho_{ij, \epsilon}^{(1)}(t,x_{c}), \partial_{x}\rho_{ij, \epsilon}^{(1)}(t,x_{c})) \partial_x^2 \rho_{ij, \epsilon}^{(1)}(t,x_{c})  \nonumber\\
   & \quad < \partial_v f_{ij}(t_c,\rho_{ij}^{(2)}(t,x_{c}), \partial_{x}\rho_{ij}^{(2)}(t,x_{c})) \partial_x^2 \rho_{ij}^{(2)}(t,x_{c}). \label{eq:fv_comp1}
\end{align}
Because the dissipation function $f_{ij}(t,u,v)$ is strictly increasing in the third argument $v$, and recalling the equivalence relations \eqref{cp_1} and \eqref{eq:cp_equality}, we have
\begin{align}
\!\!\!\! &\partial_v f_{ij}(t_c,\rho_{ij, \epsilon}^{(1)}(t,x_{c}), \partial_{x}\rho_{ij, \epsilon}^{(1)}(t,x_{c})) \nonumber \\
 & \quad =  \partial_v f_{ij}(t_c,\rho_{c}^{(2)}(t,x_{c}), \partial_{x}\rho_{ij}^{(2)}(t,x_{c})) > 0. \label{eq:fv_comp2}
\end{align}
Finally, the equality and positivity of $\partial_v f_{ij}$ terms in \eqref{eq:fv_comp2} can be used to simplify \eqref{eq:fv_comp1} to yield the simple strict inequality
\begin{align}
\partial_x^2 \rho_{ij}^{(1)}(t,x_{ij}) < \partial_x^2 \rho_{ij}^{(2)}(t,x_{ij}).
\end{align}
This contradicts \eqref{eq:cp_ineq1}, and hence our assumption must be incorrect and the proof of the lemma is complete.
\end{proof}


\begin{proof}[Proof of Lemma~\ref{lem:nodal_cross_equiv}]
Recall that by the compatibility constraints \eqref{eq:in_pressure_comp}, we have for any $k \in \partial_- j$ that $\rho_{jk}(t_c, 0) = \underline{\alpha}_{jk} (t_c,\rho_{j}(t_c))$ 
and for any $i \in \partial_+ j$ that $\rho_{ij}(t_c, L_{ij}) = \overline{\alpha}_{ij} (t_c,\rho_{j}(t_c))$.  Here $\underline{\alpha}_{jk}(t_c,\rho)$ and $\overline{\alpha}_{ij}(t_c,\rho)$ are invertible functions of $\rho$ for all $(i,j),(j,k)\in\cE$ and $\rho>0$, because we have assumed that the compatibility functions are strictly increasing for positive values in the second argument.  Let us then denote by $\underline{\alpha}_{jk,t}^{-1}(\cdot)$ and $\overline{\alpha}_{ij,t}^{-1}(\cdot)$ the inverses of the corresponding functions at the time $t$.  Then, for any $i\in\partial_+ j$ and $k\in\partial_- j$, we have the relations
\begin{align}
\rho_{ij}(t_c, L_{ij})& = \overline{\alpha}_{ij} (t_c,\underline{\alpha}_{jk,t_c}^{-1}(\rho_{jk}(t_c, 0))), \label{eq:bij1}\\
\rho_{jk}(t_c, 0) &= \underline{\alpha}_{jk} (t_c,\overline{\alpha}_{ij,t_c}^{-1}(\rho_{ij}(t_c, L_{ij}))), \label{eq:bij2}
\end{align}
where $\overline{\alpha}_{ij} (t,\underline{\alpha}_{jk,t}^{-1}(\cdot))$ and $\underline{\alpha}_{jk} (t,\overline{\alpha}_{ij,t}^{-1}(\cdot))$ are compositions of invertible increasing functions and therefore bijective and increasing.  As a result, by Definition \ref{def:first_crossing}, the existence of a crossing  point at $(t_c,x_{jk}=0)$ for some $k\in\partial_- j$ implies that there is also a crossing point at $(t_c, x_{jk} = 0)$ for all $k \in \partial_- j$ and also at at $(t_c, x_{ij} = L_{ij})$ for all $i \in \partial_+ j$.
\end{proof}

\begin{proof}[Proof of Lemma~\ref{lem:perturbed_vertex}]
We again seek to reach a contradiction by starting with the assumption that there exists a first crossing point $(t_c,x_c)$ for some $t_c \in [0,T]$ between the quantities  $\rho_{jk, \epsilon}^{(1)}(t, x_{jk})$ and
 $\rho_{jk}^{(2)}(t, x_{jk})$ for some  $k \in \partial_- j$ and $x_{jk}=x_c = 0$. By the definition of first crossing point we must have
\begin{align}
\rho_{jk, \epsilon}^{(1)}(t_c, 0) &= \rho_{jk}^{(2)}(t_c, 0), \label{cpv_1} \\
\rho_{jk, \epsilon}^{(1)}(t_c, x) &\geq \rho_{jk}^{(2)}(t_c, x), \ x \in [0,L_{jk}]. \label{cpv_2}
\end{align}
We then can apply a similar argument as that used in the proof of Lemma \ref{lem:perturbed_interior}.  Because $\rho_{jk}^{(1)}(t_c, x)  - \rho_{jk}^{(2)}(t_c, x)$ is in $C^k$,
one of the following options must be true.
\begin{itemize}
\item \emph{Option 1:}
\begin{align}
\partial_x \rho_{jk, \epsilon}^{(1)}(t_c, 0) &= \partial_x\rho_{jk}^{(2)}(t_c, 0),  \\
 \partial_x^2 \rho_{jk, \epsilon}^{(1)}(t_c, 0) &\geq \partial_x^2 \rho_{jk}^{(2)}(t_c, 0).
\end{align}
\item \emph{Option 2:}
\begin{align}
\partial_x \rho_{jk, \epsilon}^{(1)}(t_c, 0) &> \partial_x\rho_{jk}^{(2)}(t_c, 0). \label{eq:cpv_inequality}
\end{align}
\end{itemize}

By following the exact same argument as in the proof of Lemma~\ref{lem:perturbed_interior},  we can show that \emph{Option 1} leads to a contradiction. What remains is to prove that \emph{Option 2} is also disallowed.  Because $f_{jk}(t,u,v)$ is strictly increasing in $v$, using \eqref{cpv_1} and \eqref{eq:cpv_inequality} we see that
\begin{align}
f_{jk}(t_c, \rho_{jk}^{(1)}(t_c,0), &\partial_x \rho_{jk}^{(1)}(t_c,0))  \nonumber\\
& > f_{jk}(t_c, \rho_{jk}^{(2)}(t_c,0), \partial_x \rho_{jk}^{(2)}(t_c,0)). \label{eq:pert_ineq1}
\end{align}
Applying Lemma \ref{lem:nodal_cross_equiv} to edges outgoing from node $j$ we find that \eqref{cpv_1} and \eqref{eq:cpv_inequality} hold for all $k \in \partial_- j$, hence so does \eqref{eq:pert_ineq1}.  Combining with the dissipation equation \eqref{eq:in_dissipation_eq}, this gives for all $k \in \partial_- j$  that $\phi_{jk}^{(1)}(t_c, 0) < \phi_{jk}^{(2)}(t_c, 0)$.  Similarly, Lemma \ref{lem:nodal_cross_equiv} implies that the relations $\rho_{ij, \epsilon}^{(1)}(t_c, L_{ij}) = \rho_{ij}^{(2)}(t_c, L_{ij})$ and $\partial_x \rho_{ij, \epsilon}^{(1)}(t_c, L_{ij}) > \partial_x\rho_{ij}^{(2)}(t_c, L_{ij})$ must hold for all $i\in\partial_+ j$, and hence $\phi_{ij}^{(1)}(t_c, L_{ij}) < \phi_{ij}^{(2)}(t_c, L_{ij})$ hold for all $i \in \partial_+ j$ as well.  We then apply the flow conservation equation \eqref{eq:in_nodal_continuity} to obtain
\begin{align}
q^{(1)}(t_c) &= \sum_{k \in \partial_- j} \phi_{jk}^{(1)}(t_c, 0) +  \sum_{i \in \partial_+ j} \phi_{ij}^{(1)}(t_c, L_{ij}) \nonumber \\
& <   \sum_{k \in \partial_- j} \phi_{jk}^{(2)}(t_c, 0) +  \sum_{i \in \partial_+ j} \phi_{ij}^{(2)}(t_c, L_{ij})  \nonumber \\
&= q^{(2)}(t_c).
\end{align}
The last statement is in contradiction with the assumptions of Lemma~\ref{lem:perturbed_vertex}.
\end{proof}

Before proceeding to the proof of Theorem~\ref{theorem} we state one last lemma that relates the solution of the perturbed system in  \eqref{perturbed1}-\eqref{perturbed2}  to the original system.
\begin{lemma} \label{rem:perturbed_ordering}
The solution to the perturbed system $\rho_{ij,\epsilon}^{(1)}(t,x_{ij})$ is always greater than or equal to the solution $\rho_{ij}^{(1)}(t,x_{ij})$ of the original system for all $t \in [0,T]$.
\end{lemma}
\begin{proof}
We observe that all assumptions in Lemma~\ref{lem:perturbed_interior} and Lemma~\ref{lem:perturbed_vertex} are satisfied if we replace $\rho_{ij}^{(2)}(t,x_{ij})$ by  $\rho_{ij}^{(1)}(t,x_{ij})$.
As a consequence, there can be no first crossing point between $\rho_{ij,\epsilon}^{(1)}(t,x_{ij})$ and $\rho_{ij}^{(1)}(t,x_{ij})$. Therefore, for all $t \in [0,T]$ we must have $\rho_{ij,\epsilon}^{(1)}(t,x_{ij}) \geq \rho_{ij}^{(1)}(t,x_{ij})$.
\end{proof}
\vspace{0.1in}

\subsection{Proof of the Main Result}

\begin{proof}[Proof of Theorem~\ref{theorem}]
Fix an $\epsilon > 0$.  Let $\rho_{ij, \epsilon}^{(1)}(t,x_{ij})$ be the solution to the perturbed system in  \eqref{perturbed1}-\eqref{perturbed2}
with nodal input parameters set at $q_{i}^{(1)}(t)$. Then by Lemma~\ref{lem:perturbed_interior}, there can be no first crossing point between $\rho_{ij, \epsilon}^{(1)}(t,x_{ij})$ and $\rho_{ij}^{(2)}(t,x_{ij})$ such that $t_c \in [0,t_0]$ and $0 < x_{ij} < L_{ij}$
for some $(i,j) \in \cE$. The above statement is also true by Lemma~\ref{lem:perturbed_vertex} for $i \in \mathcal{S}$.
Further, by Lemma~\ref{rem:perturbed_ordering}, we have that $\rho_{i}^{(1)}(t) \geq \rho_{i}^{(2)}(t)$ implies $\rho_{i, \epsilon}^{(1)}(t) \geq \rho_{i}^{(2)}(t)$ and hence there is no crossing point at $i \notin \mathcal{S}$. This means for all $t \in [0,t_0]$ we have
\begin{align}
\rho_{ij, \epsilon}^{(1)}(t,x_{ij}) \geq \rho_{ij}^{(2)}(t,x_{ij}), \label{eq:main_ineq}
\end{align}
for all $(i,j) \in \cE$ and $x_{ij} \in I_{ij}$. Because $\epsilon > 0$ was chosen arbitrarily, we can take the limit $\epsilon \rightarrow 0$ in \eqref{eq:main_ineq}, and using Assumption~\ref{assume}-(iv), we get
\begin{align}
\rho_{ij}^{(1)}(t,x_{ij}) \geq \rho_{ij}^{(2)}(t,x_{ij}), \label{eq:main_ineq}
\end{align}
This completes the proof of the Main Theorem.
\end{proof}

\section{Discussion and Applications} \label{sec:application}

The MOP properties established in Sections \ref{sec:formulation} and \ref{sec:proof} have several important interpretations for understanding the possible robust optimal control formulations for parabolic PDEs on metric graphs. We consider robust control formulations where the nodal parameter functions $q_j(t)$ are prescribed ahead of time within a compact subset of $C^k[0,T]$, but are uncertain.  Such control formulations appear in problems related to the transportation of commodities over networks, in particular the flow of
compressible fluids such as natural gas in large scale pipeline systems \cite{herty10,zlotnik15cdc}, where the flows are well-described by parabolic systems of PDE of the form \eqref{eq:in_continuity}-\eqref{eq:in_nodal_continuity}.  In order to solve robust optimal control problems for such systems, computationally tractable formulations are essential. Consider the following deterministic optimal control problem:
\begin{align}
\!\!\! \min \,\,\, & \cJ(\rho, \phi, \alpha) = \int_0^T \cL(t,\underline{\phi}(t),\overline{\phi}(t),\underline{\alpha}(t),\overline{\alpha}(t)) dt, \label{eq:det_obj} \\
\!\!\!\mbox{s.t.} \,\,\, & \dS \pp_t\rho_{ij}(t,x_{ij})+\pp_x\phi_{ij}(t,x_{ij})  =  0  \\
\!\!\! &\phi_{ij}(t,x_{ij})+f_{ij}(t,\rho_{ij}(t,x_{ij}), \partial_{x}\rho_{ij}(t,x_{ij}))  =0, \\
\!\!\! &\underline{\rho}_{ij}(t) \!=\! \underline{\alpha}_{ij}\rho_{i}(t), \,\, \overline{\rho}_{ij}(t)  \!=\! \overline{\alpha}_{ij}\rho_{j}(t), \, \fA (i,j)\in\cE \\
\!\!\! &q_j(t)+\sum_{i\in\partial_{+}j}\overline{\phi}_{ij}- \sum_{k\in\partial_{-}j}\underline{\phi}_{jk}=0, \quad \fA j\in\cV \\
\!\!\! & \rho_{min} \leq \rho_{ij}(t,x) \leq \rho_{max}, \,\, \fA (i,j)\in\cE . \label{eq:det_ineq}
\end{align}
In the above formulation, the density compatibility functions are linear with factors $\underline{\alpha}_{ij}$ and $\overline{\alpha}_{ij}$.  We also define $\underline{\phi}(t)=(\underline{\phi}_{\pi^{-1}(1)}(t),\ldots,\underline{\phi}_{\pi^{-1}(E)}(t))\in\bR^E$ and $\overline{\phi}(t)=(\overline{\phi}_{\pi^{-1}(1)}(t),\ldots,\overline{\phi}_{\pi^{-1}(E)}(t)) \in \bR^{E}$, where $\pi:\cE\to[E]$ is a mapping from the set of edges to the integers $1$ to $E$.  Similarly, we define $\underline{\alpha}(t)=(\underline{\alpha}_{\pi^{-1}(1)}(t),\ldots,\underline{\alpha}_{\pi^{-1}(E)}(t))\in\bR^E$ and $\overline{\alpha}(t)=(\overline{\alpha}_{\pi^{-1}(1)}(t),\ldots,\overline{\alpha}_{\pi^{-1}(E)}(t))\in\bR^E$, which form the collection of control functions.

A version of problem \eqref{eq:det_obj}-\eqref{eq:det_ineq} can be formulated such that the solution is robust to uncertain variation in the nodal parameter functions $q_i(t)$ within some known bounds, i.e.,
\begin{align}
q^{(1)}_j(t) \geq q_j(t) \geq q^{(2)}_j(t),  \quad \fA j\in\cV \text{ and } t\in[0,T].  \label{box}
\end{align}
The resulting problem is extremely challenging because of the semi-infinite set of constraints.  Using Theorem~\ref{theorem}, however, we can obtain significantly simplified reformulations of the robust control problem as well as a ``monitoring'' mechanism that we describe in following subsections.

\subsection{Simplified Representation of Robust Optimal Control}

As a consequence of Theorem~\ref{theorem}, we can obtain a reformulation of the semi-infinite constrained robust control problem with interval uncertainty as in \eqref{box} by enforcing feasibility only for the extreme scenarios.  As long as the optimal control solution satisfies the constraints for the two extremal cases of nodal parameter functions $q_j(t)$ for $j\in\cV$, feasibility will also be guaranteed for all nodal parameter functions that are bounded by the extreme scenarios.
We rewrite the entire formulation below for completeness.
\begin{align}
\!\!\! \min \,\,\, & \cJ(\rho, \phi, \alpha) \! =\!\! \int_0^T \!\!\!\!\cL(t,\underline{\phi}(t),\overline{\phi}(t),\underline{\alpha}(t),\overline{\alpha}(t)) dt, \label{eq:rob_obj} \\
\!\!\!\mbox{s.t.} \,\,\, & \dS \pp_t\rho_{ij}(t,x_{ij})+\pp_x\phi_{ij}(t,x_{ij})  =  0 \label{eq:nom_a} \\
\!\!\! &\phi_{ij}\!(t,x_{ij}) \!+\! f_{ij}(t,\rho_{ij}(t,x_{ij}), \partial_{x}\rho_{ij}(t,x_{ij})) \!=\! 0, \\
\!\!\! &\underline{\rho}_{ij}\!(t) \!=\! \underline{\alpha}_{ij}\rho_{i}\!(t), \,\, \overline{\rho}_{ij}\!(t)  \!=\! \overline{\alpha}_{ij}\rho_{j}\!(t), \, \fA \! (\!i\!,\!j\!)\!\in\!\cE \\
\!\!\! &\hat{q}_j(t)+\sum_{i\in\partial_{+}j}\overline{\phi}_{ij}- \sum_{k\in\partial_{-}j}\underline{\phi}_{jk}=0, \quad \fA j\in\cV \label{eq:nom_b} \\
\!\!\! & \dS \pp_t\rho_{ij}^{(1)}(t,x_{ij})+\pp_x\phi_{ij}^{(1)}(t,x_{ij})  =  0 \label{eq:sol1_a} \\
\!\!\! &\phi_{ij}^{(1)}\!(t,x_{ij}) \!+\! f_{ij}(t,\rho_{ij}^{(1)}(t,x_{ij}), \partial_{x}\rho_{ij}^{(1)}(t,x_{ij})) \!=\! 0, \\
\!\!\! &\underline{\rho}_{ij}^{(1)}\!(t) \!=\! \underline{\alpha}_{ij}\rho_{i}^{(1)}\!(t), \,\, \overline{\rho}_{ij}^{(1)}\!(t)  \!=\! \overline{\alpha}_{ij}\rho_{j}^{(1)}\!(t), \, \fA \! (\!i\!,\!j\!)\!\in\!\cE  \\
\!\!\! &q_j^{(1)}(t)+\sum_{i\in\partial_{+}j}\overline{\phi}_{ij}^{(1)}- \sum_{k\in\partial_{-}j}\underline{\phi}_{jk}^{(1)}=0, \quad \fA j\in\cV \label{eq:sol1_b} \\
\!\!\! & \dS \pp_t\rho_{ij}^{(2)}(t,x_{ij})+\pp_x\phi_{ij}^{(2)}(t,x_{ij})  =  0 \label{eq:sol2_a} 
\end{align}
\begin{align}
\!\!\! &\phi_{ij}^{(2)}\!(t,x_{ij}) \!+\! f_{ij}(t,\rho_{ij}^{(2)}(t,x_{ij}), \partial_{x}\rho_{ij}^{(2)}(t,x_{ij})) \!=\! 0, \\
\!\!\! &\underline{\rho}_{ij}^{(2)}\!(t) \!=\! \underline{\alpha}_{ij}\rho_{i}^{(2)}\!(t), \,\, \overline{\rho}_{ij}^{(2)}\!(t)  \!=\! \overline{\alpha}_{ij}\rho_{j}^{(2)}\!(t), \, \fA \! (\!i\!,\!j\!)\!\in\!\cE \\
\!\!\! &q_j^{(2)}(t)+\sum_{i\in\partial_{+}j}\overline{\phi}_{ij}^{(2)}- \sum_{k\in\partial_{-}j}\underline{\phi}_{jk}^{(2)}=0, \quad \fA j\in\cV \label{eq:sol2_b}\\
\!\!\! & \rho_{min} \leq \rho_{ij}^{(1)}(t,x_{ij}) \leq \rho_{max}, \,\, \fA (i,j)\in\cE \label{eq:rob_ineq1} \\
\!\!\! & \rho_{min} \leq \rho_{ij}^{(2)}(t,x_{ij}) \leq \rho_{max}, \,\, \fA (i,j)\in\cE. \label{eq:rob_ineq2}
\end{align}

In the above formulation, equations \eqref{eq:nom_a}-\eqref{eq:nom_b}, \eqref{eq:sol1_a}-\eqref{eq:sol1_b}, and \eqref{eq:sol2_a}-\eqref{eq:sol2_b} represent the physical constraints for the nominal, high injection, and low injection cases, respectively.  For evaluating the objective function $\cJ$, we chose the flow solutions corresponding to the nominal injection profiles $\hat{q}_i(t)$ that are bounded by the extremal envelopes.  We  enforce feasibility of the $\rho_{ij}(t,x_{ij})$ only with respect to the extreme scenarios corresponding to the lower and upper envelopes $q_i^{(1)}(t)$ and $q_i^{(2)}(t)$ of the uncertain nodal parameters $q_i(t)$. By Theorem~\ref{theorem}, as long as $q_j^{(1)}(t) \geq \hat{q}_j(t) \geq q_j^{(2)}(t)$, then the corresponding densities $\rho_{ij}(t,x_{ij})$ must also satisfy $\rho_{ij}^{(1)}(t,x_{ij}) \geq \rho_{ij}(t,x_{ij}) \geq \rho_{ij}^{(2)}(t,x_{ij})$ for all $t\in[0,T]$, and hence the constraints \eqref{eq:rob_ineq1} and \eqref{eq:rob_ineq2} as well.

If the objective function is also monotone with respect to the nodal input parameters, one can also obtain a simplified representation of min-max robust optimal control, as in \cite{vuffray15cdc}.

\subsection{Policy for Real-Time Nodal Monitoring}
We have obtained a tractable formulation for the robust control problem when the uncertainty envelopes of the nodal input parameters are known a priori.  Suppose that there exists a feasible solution to the simplified representation for the robust optimal control problem in \eqref{eq:rob_obj}-\eqref{eq:rob_ineq2}, and that the optimal control vectors $\underline{\alpha}(t)$ and $\overline{\alpha}(t)$ that keep the system feasible under uncertainty have been obtained.  In this section we suggest a simple monitoring mechanism that can be used to react to real-time deviations outside of the predicted uncertainty envelope.  If, for instance, an error in uncertainty quantification causes $q_i(t)$ for some $i \in \cV$ to deviate outside of the predicted envelope $[q_i^{(2)}(t),q_i^{(1)}(t)]$, then feasibility is no longer guaranteed with the control solution $\underline{\alpha}(t)$ and $\overline{\alpha}(t)$ to problem \eqref{eq:rob_obj}-\eqref{eq:rob_ineq2} because the assumptions of
Theorem~\ref{theorem} no longer hold.  One way to compensate for such variation is to enforce the nodal input parameters to the predicted upper or lower bounds, $q_i^{(1)}(t)$ or $q_i^{(2)}(t)$, as appropriate. However, this may be too conservative for enforcing the density constraints \eqref{eq:det_ineq}, because if $q_i(t)$ is not within the expected envelope, then it does not necessarily mean that the density constraints, which are the focus in the motivating applications, are violated. However, Theorem~\ref{theorem} still applies, and this facilitates a much less conservative Nodal Monitoring Policy (NMP) to reactively maintain system densities within feasible bounds.

\vspace{0.1in}
\noindent \emph{ \bf Nodal Monitoring Policy (NMP):} Let $\overline{\cal{S}} \subset \cV$ be the subset of nodes where the upper bound $q_i(t) \leq q_i^{(1)}(t)$ on nodal parameters is violated, and let $\underline{\cal{S}} \subset \cV$ be the subset of nodes where the lower bound $q_i(t) \geq q_i^{(2)}(t)$ is violated. Let $\rho_i^{(1)}(t)$ and $\rho_i^{(2)}(t)$ for $i\in\cV$ be the collections of nodal density solutions corresponding to fixing the nodal input parameters at
$q_i^{(1)}(t)$ and $q_i^{(2)}(t)$, respectively. The policy is to monitor the real-time density profiles $\rho_i(t)$ for $i \in \underline{\cal{S}} \cup \overline{\cal{S}}$. If no crossing points are encountered between the real-time solution $\rho_i(t)$ and the upper and lower density profile solutions $\rho_i^{(1)}(t)$ and $\rho_i^{(2)}(t)$, respectively, then the system is safe with respect to the density limits. Alternatively, suppose we encounter a crossing point at time $t_c$ at node $i \in \overline{\cal{S}}$. The policy then calls to reset the nodal parameter at $i$ to $q_i^{(1)}(t)$ and to leave the rest of the nodal parameters untouched.  It turns out that this simple action is sufficient to guarantee that the system-wide density remains within $\rho_{\min}$ and $\rho_{\max}$. This is contained in the following corollary.

 \begin{corollary}[Sufficiency of Nodal Monitoring Policy] Suppose we implement the NMP described above. Then we have $\rho^{(2)}(t,x_{ij}) \leq \rho(t,x_{ij}) \leq \rho^{(1)}(t, x_{ij})$ for all $t \in [0,T]$ and $(i,j) \in \cE$.
\end{corollary}
\begin{proof} The proof is a direct application of Theorem~\ref{theorem}, because by construction of the policy, the assumptions in the theorem are satisfied for all $t \in [0,T]$.
\end{proof}

\section{Conclusions} \label{sec:conc}

In this manuscript, we have derived monotone order propagation (MOP) properties for a class of nonlinear parabolic partial differential equation (PDE) systems on metric graphs.  We have established an {\em ab initio} proof that ordering properties of the solution to the initial value problem (IVP) are preserved when the initial conditions and time-varying nodal parameters at vertices are appropriately ordered.  In addition, we proved that when monotone ordering is not preserved, the first crossing of solutions occurs at a graph vertex.  These results have important implications for robust optimal control of subsonic compressible fluid flow with energy dissipation on physical networks subject to uncertainty.  In particular, there exists a direct application to robust dynamic optimization of compressors in large-scale natural gas pipeline networks \cite{zlotnik15cdc,vuffray15cdc,zlotnik16ecc}, and energy systems in general \cite{budivsic12}.   In addition to simplified robust optimal control formulations, we have presented a nodal monitoring policy (NMP) for control of nonlinear parabolic PDE on graphs subject to parameter uncertainty.  The results may also find uses in analysis of vehicle transportation networks \cite{dapice08,ma04,tjandra03phd,gottlich05,gugat05,herty03,como10}, and information systems \cite{como13a,como13b,intanagonwiwat02}.

\section{Acknowledgements}
The authors would like to thank Giacomo Como and Michael Herty for insightful discussions on monotone order properties and PDEs on networks.  This work was carried out under the auspices of the National Nuclear Security Administration of the U.S. Department of Energy at Los Alamos National Laboratory under Contract No. DE-AC52-06NA25396, and was partially supported by DTRA Basic Research Project \#10027-13399 and by the Advanced Grid Modeling Program in the U.S. Department of Energy Office of Electricity.





\bibliographystyle{unsrt}
\bibliography{gas_master,monotonicity_master}

\end{document}